\documentclass{amsart}
\usepackage{tikz}
\usepackage{amssymb}

\newtheorem{theorem}{Theorem}
\newtheorem{lemma}{Lemma}
\newtheorem{cor}{Corollary}
\newtheorem{problem}{Problem}

\newcommand{\deq}{\mathrel{\mathop:}=}
\newcommand{\alg}[1]{\mathbf{#1}}
\newcommand{\var}[1]{\mathsf{#1}}
\newcommand{\sipalg}[1]{\overline{\mathbf{#1}}}
\newcommand{\Bnalg}[1]{\overline{\mathbf{B}}_{#1}}

\newcommand{\dw}[1]{\mathord{\downarrow}{#1}}
\newcommand{\up}[1]{\mathord{\uparrow}{#1}}

\tikzstyle{every label}=[label distance=0pt]
\tikzstyle{bdot}[1.5]=[circle,fill,draw,thick,minimum size=#1mm,inner sep=0pt]
\tikzstyle{dot}[1.5]=[circle,draw,thick,minimum size=#1mm,inner sep=0pt]
\tikzstyle{every edge}=[draw=black,thick]

\title{Decidable quasivarieties of p-algebras} 

\author[Kowalski]{Tomasz Kowalski$^{1,2,3}$}
\author[Słomczyńska]{Katarzyna Słomczyńska$^{4}$}

\address{$^{1}$ Department of Logic, Jagiellonian University}
\email{tomasz.s.kowalski@uj.edu.pl}

\address{$^{2}$ Department of Physical and Mathematical Sciences, La Trobe
  University} 
\email{t.kowalski@latrobe.edu.au}

\address{$^{3}$ School of Historical and Philosophical Inquiry, The University of Queensland} 
\email{t.kowalski@uq.edu.au}

\address{$^{4}$ Department of Mathematics, Pedagogical University, Kraków}
\email{irena.korwin-slomczynska@up.edu.pl}

\begin{document}

\maketitle

\begin{abstract}
We show that for quasivarieties of p-algebras the properties of (i) having decidable
first-order theory and (ii) having decidable first-order theory of the finite
members, coincide. The only two quasivarieties with these properties are
the trivial variety and the variety of Boolean algebras. This contrasts sharply,
even for varieties, with the situation in Heyting algebras where decidable varieties
do not coincide with finitely decidable ones. 
\end{abstract}

\section{Introduction} 

As usual, we write
$\mathrm{Th}(\mathcal{K})$ for the
first-order theory of a class $\mathcal{K}$, and $\mathcal{K}_\mathrm{fin}$
for the set of finite objects from $\mathcal{K}$. For 
varieties of Heyting algebras a complete classification of both properties
has been known for more than three decades.
Burris~\cite{Bur82} showed that for a variety $\mathcal{V}$
of Heyting algebras $\mathrm{Th}(\mathcal{V})$ is decidable if and only if
$\mathcal{V}$ is contained in the variety $\var{BA}$ of Boolean algebras. 
Idziak and Idziak~\cite{II88} completed the picture by showing that
$\mathrm{Th}(\mathcal{V}_\mathrm{fin})$ is decidable if and only if
$\mathcal{V}$ is contained in the variety $\var{LC}$ of \emph{linear}
Heyting algebras. Idziak~\cite[Thm.~3.5]{Idz89a}
and~\cite[Thm.~4]{Idz89b}
give very general sufficient conditions for finite undecidability in
congruence distributive \emph{varieties}, namely congruence non-linearity and
congruence non-permutability. The proofs of these 
results heavily involve congruences, so they are
not applicable to quasivarieties. Thus, the corresponding (un)decidability
questions for quasivarieties are in general open. 

A close relative of Heyting algebras is the variety $\var{Pa}$ of
\emph{p-algebras} (also known in the literature as
distributive p-algebras, or pseudo-complemented distributive lattices).
P-algebras are the class of 
bounded distributive lattices with a unary operation ${}^*$ satisfying the
condition
$$
x\wedge y = 0 \iff y\leq x^*
$$
which is in fact equivalent to a finite set of equations, so $\var{Pa}$ is 
indeed a variety. An equational base is given for example in
Bergman~\cite[Ch.~3.4]{Ber11}, but it was certainly known to
Lee~\cite{Lee70}, who completely described varieties of p-algebras.  
They form a chain
$\var{Pa}_{-1}\subset \var{Pa}_0\subset \dots\subset\var{Pa}_i\subset\dots
\subset\var{Pa}$ of type $\omega+1$, where $\var{Pa}_{-1}$ is the trivial variety,
$\var{Pa}_{i}$, for $i\in \omega$, is the variety generated by
the algebra $\Bnalg{i}$, and $\var{Pa}$ is the variety
of all p-algebras. The algebra $\Bnalg{i}$, as a poset, is 
the finite Boolean algebra $\alg{B}_i$ on $i$ atoms, with a new top element
added. More precisely, given any Boolean algebra $\alg{B} = (B;\wedge,\vee,\neg,0,e)$,
with the top element $e$, we put $\overline{B} = B\uplus\{1\}$, extend the    
the order on $B$ by requiring $1>x$ for all $x\in B$, extend the lattice
operations accordingly, and define pseudo-complementation by
$$
x^*\deq\begin{cases}
         \neg x & \text{ if } x\in B\setminus\{0\}\\
         1 & \text{ if } x = 0\\
         0 & \text{ if } x = 1
       \end{cases}     
$$
to get the p-algebra $\sipalg{B} = (B;\wedge,\vee,{}^*,0,1)$.
A p-algebra is subdirectly irreducible if and only if it is of the form
$\sipalg{B}$ for some Boolean algebra $\alg{B}$.

The lattice of
subquasivarieties of $\var{Pa}$ is much more complex, in fact
Wroński~\cite{Wro76} showed that it is uncountable, and
Gr\"atzer~\emph{et al.}~\cite{GLQ80} improved it by showing the same for
$\var{Pa}_3$.
From the point of view of logic, p-algebras are the algebraic semantics,
in the sense of Rebagliato and Verd\'u~\cite{RV93}, of
the implication-free fragment of the intuitionistic propositional calculus
with negation. For a quick 
introduction-cum-survey of p-algebras we refer the reader to
Bergman~\cite[Ch.~3.4]{Ber11}. Our notation mostly comes from there.

We will characterise the quasivarieties of
p-algebras with decidable and with finitely decidable first-order theories.
It turns out that, unlike in varieties of Heyting algebras, these classes coincide. 
An interesting technicality we encounter on the way is an observation
concerning the algebras whose underlying lattice is a three-element chain.   
This algebra admits a unique Heyting implication and the resulting Heyting
algebra, often denoted by $\alg{H}_3$, is the unique three-element Heyting
algebra. Its p-algebra reduct is isomorphic to $\Bnalg{1}$, and it is the unique 
three-element p-algebra. Now, 
although the variety $V(\alg{H}_3)$ of Heyting algebras
has finitely decidable first-order theory, the quasivariety
$Q(\Bnalg{1})$ has finitely undecidable theory.
Since $Q(\Bnalg{1})$ is a variety, namely,
the variety of \emph{Stone algebras}, we have a curious contrast between
finite decidability of the variety generated by $\alg{H}_3$ and
finite undecidability of the variety generated by a reduct of the same algebra.

We assume familiarity with the method of
\emph{semantic embeddings}; we refer the reader to
Burris and Sankappanavar~\cite{BS81} Ch.~V, {\S}5 for details. Our proofs
proceed by modifications of proofs from Burris~\cite{Bur82} and
Idziak and Idziak~\cite{II88}, so the
reader should keep these articles handy to refer to them when needed.

Priestley~\cite{Pri75} gives a topological duality for p-algebras, of which we
will only need the topology-free part for finite algebras. 
Dual objects are finite posets. Morphisms (henceforth \emph{pp-morphisms}, to
distinguish them from p-morphisms, familiar from modal and intuitionistic logic)
are order preserving maps
$f\colon P\to Q$ satisfying $f(\max\up{p}) = \max\up{f(p)}$ for every $p\in P$,
where $\up{p}$ is the upward closure of $\{p\}$.
For a finite poset $P$, let $\mathrm{Up}(P)$ be the set of upsets of $P$,
ordered by inclusion. Then $\mathrm{Up}(P)$ is a lattice under the natural set
operations, and a p-algebra under the operation $X^* = P\setminus \dw{X}$  
where $\dw{X}$ is the downward closure of $X$. 
Every finite p-algebra $\alg{A}$ is isomorphic to the p-algebra
$\mathrm{Up}(\mathcal{J}(\alg{A}))$ where $\mathcal{J}(\alg{A})$
is the poset of join-irreducible elements
of $\alg{A}$ with the dual order inherited
from $\alg{A}$ (recall that in the finite case
$\mathcal{J}(\alg{A})$ is order-isomorphic to the poset of prime filters of $\alg{A}$
ordered by inclusion). The usual duality between surjective/injective pp-morphisms
of posets and injective/surjective homomorphisms of p-algebras holds, as well as
the duality between finite disjoint unions and finite direct products.

\section{Decidable quasivarieties}

Burris~\cite{Bur82} proved that the variety of Heyting algebras generated
by the three element algebra $\alg{H}_3$ has undecidable theory. The proof
proceeds via semantically embedding Boolean Pairs (henceforth
$\mathsf{BP}$) into powers of $\mathbf{H}_3$, so in fact it also shows that
$Q(\alg{H}_3)$ has undecidable theory.
The p-algebra reduct of
$\mathbf{H}_3$ is $\Bnalg{1}$, but the proof from~\cite{Bur82} does not apply directly,
since it relies on the presence of Heyting implication. With a tiny modification
below, it works. Recall that a Boolean Pair is a structure
$(B,B_0)$, where  $B$ is a Boolean algebra and
$B_0$ is a Boolean subalgebra of $B$. Formally,
$\mathsf{BP}$ is the class of Boolean algebras expanded by a unary
predicate $B_0$ whose interpretation is a subalgebra of $B$.
Rubin~\cite{Rub76} shows that the theory of $\mathsf{BP}$ is undecidable.

\begin{lemma}\label{lem:hered-undec}
There is a semantic embedding of $\mathsf{BP}$\/ into   
$Q(\Bnalg{1})$. Hence,
$\mathrm{Th}(Q(\Bnalg{1}))$ is undecidable.
\end{lemma}
  
\begin{proof}
An arbitrary Boolean Pair $(B, B_0)$ can be presented as
a Boolean algebra  $B$ of subsets of some $I$, and
a subalgebra $B_0$ is of $B$, so that $B_0\leq B\leq \mathcal{P}(I)$.
Let the universe of $\Bnalg{1}$ be $\{0,e,1\}$, and define
$$
P(B, B_0) = \{f\in \{0,e,1\}^I: f^{-1}(0)\in B_0, f^{-1}(1)\in B\}.
$$ 
Then $P(B, B_0)$ is a subuniverse of $\Bnalg{1}^I$. Take a
characteristic function $\chi_X$ of subsets of $I$ given by
$$
\chi_X(i) = \begin{cases}
              1 & \text{ if } i\in X\\
              e & \text{ if } i\notin X
           \end{cases}       
$$
and put
$C = \{\chi_X: X\in B\}$ and $C_0 = \{\chi_X: X\in B_0\}$.  
Clearly $(C,C_0)$ with $C_0\leq C\leq \{1,e\}^I$ is a Boolean Pair isomorphic to
$(B,B_0)$. It remains to show that $(C,C_0)$ can be defined in $P(B, B_0)$
using the language of p-algebras expanded by a single constant $\overline{e}$,
interpreted as identically equal to $e$. To this end, it suffices to show that
\begin{align*}
C &= \{f\vee\overline{e}: f\in P(B, B_0)\},\\
C_0 &= \{f^{**}\vee \overline{e}: f\in P(B, B_0)\} 
\end{align*}
since the operations in $C$ and $C_0$ are relativisations
of meet, join, and pseudo-complement to $\{1,e\}^I$. 
The inclusions from left to right are straightforward.
The converses are not difficult either, so let us just show the second one.
By closure under complementation, $X\in B_0$ iff $I\setminus X\in B_0$,
so consider $X = I\setminus f^{-1}(0)$. 
If $i\in X$, then $f(i) \neq 0$ and therefore $f(i)^{**} = 1$, so 
$f(i)^{**}\vee e = 1 = \chi_X(i)$. If $i\notin X$, then $f(i) = 0$
and therefore $f(i)^{**} = 0$, so $f(i)^{**}\vee e = e = \chi_X(i)$.

It follows that $\var{BP}$ semantically embeds into  $Q(\Bnalg{1})$. 
\end{proof}  

Since every quasivariety $\mathcal{Q}$ of p-algebras has
$\mathcal{Q}\subseteq\var{BA}$ if and only if $\Bnalg{1}\notin\mathcal{Q}$
we immediately obtain the next result.

\begin{cor}\label{cor:dec-qvar}
A quasivariety $\mathcal{Q}$ of p-algebras has decidable first-order theory
if and only if $\mathcal{Q}\subseteq\var{BA}$.
\end{cor}  

\section{Finitely decidable quasivarieties}

Now we turn to theories of finite algebras. As we mentioned already,
Theorem~4 of Idziak~\cite{Idz89b} applies
to all varieties of p-algebras, but since its proof makes an
essential use of congruences, it is not readily applicable to quasivarieties.
Fortunately, some modifications of ideas from Idziak and Idziak~\cite{II88}
apply. Let $\alg{N}$ be the algebra whose underlying lattice is depicted
in~Figure~\ref{fig:N}.
\begin{figure}
\begin{center}
\begin{tikzpicture}
\node[dot, label=below:$0$] (v0) at (0,0) {};
\node[dot, label=right:$a$] (v1) at (0,0.5) {};
\node[dot, label=left:$b$] (v2) at (-0.5,1) {};
\node[dot, label=right:$c$] (v3) at (0.5,1) {};
\node[dot, label=right:$e$] (v4) at (0,1.5) {};
\node[dot, label=above:$1$] (v5) at (0,2) {};
\path[draw,thick] (v0)--(v1)--(v2)--(v4)--(v5);
\path[draw,thick] (v1)--(v3)--(v4);
\end{tikzpicture}
\end{center}\label{fig:N}\caption{The algebra $\mathbf{N}$.}
\end{figure}
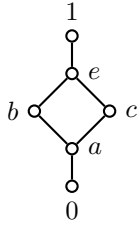
Clearly $\alg{N}$ can be viewed as a Heyting algebra, and hence also as a
p-algebra. Note that $\alg{N}$ as a p-algebra is not subdirectly
irreducible, indeed it has three minimal meet-irreducible congruences, namely,
$\theta(b,1)$, $\theta(c,1)$, and $\theta(a,e)$ the last of which is not
a Heyting algebra congruence. The lemma below is not strictly necessary
for the proof of our finite undecidability result below, but it clarifies both the
connection to and the discrepancy with the finite decidability result
of Idziak and Idziak~\cite{II88}

\begin{lemma}\label{lem:char-stone}
Let $\mathcal{Q}$ be a quasivariety of p-algebras. The following are equivalent.
\begin{enumerate}
\item $\mathcal{Q}\subseteq \var{BA}$.
\item $\Bnalg{1}\notin\mathcal{Q}$.
\item $\alg{N}\notin\mathcal{Q}$.
\end{enumerate}  
\end{lemma}  

\begin{proof}
The equivalence of (1) and (2) is well known. To prove the equivalence of (2) and (3)
we will show that $Q(\alg{N}) = Q(\Bnalg{1})$. It is easy to see that
$\Bnalg{1}$ is a subalgebra of $\alg{N}$ with the universe $\{0,e,1\}$,
so $Q(\alg{N})\supseteq Q(\Bnalg{1})$.
For converse, consider $\Bnalg{1}^3$ and the set
$N = \{(0,0,0), (e,e,e), (1,e,e), (e,e,1), (1,e,1), (1,1,1)\}$.  Simple
calculations show that $N$ is a subuniverse of $\Bnalg{1}^3$ and
the subalgebra on this universe is isomorphic to $\alg{N}$. Hence
$\alg{N}\in Q(\Bnalg{1})$ proving that $Q(\alg{N})\subseteq Q(\Bnalg{1})$.
\end{proof}  

Idziak and Idziak~\cite{II88} proved that finite graphs can be semantically embedded
into the variety of Heyting algebras generated by $\alg{N}$, or indeed, in any
class of Heyting algebras containing all finite powers of $\alg{N}$.
But $\alg{N}$ viewed as a Heyting algebra is subdirectly irreducible, and
does not belong to $V(\alg{H}_3)$, in stark contrast to
Lemma~\ref{lem:char-stone}. This is the root of the discrepancy between the
results in~\cite{II88} and ours. Inspection
of the proof from~\cite{II88} shows that it only relies on lattice properties of
$\alg{N}$, so in principle it can be applied to our case as well.
However, we give a simpler and more direct proof. 

\begin{lemma}\label{lem:fingraph-int}
The class of all finite graphs can be interpreted in
$Q_{\mathrm{fin}}(\Bnalg{1})$. 
\end{lemma}

\begin{proof}
We use duality for finite p-algebras.
Take a finite graph $G=(V,E)$, with $E$ viewed as a set of 2-element subsets of
$V$. Let $P_G$ be the poset with the universe $\{\emptyset\}\cup\{\{v\}: v\in V\}
\cup E$ ordered by reverse inclusion. Put $I = E\cup V$ and
endow $I$ with the identity ordering so that $I$ becomes an antichain.
Let $\mathbf{2}$ be the two-element chain $\{0,1\}$
under the natural ordering. Let $U = \mathbf{2}\times I$, as a poset.
The map $f\colon U\to P_G$ given by
$f(\langle 1,i\rangle) = \emptyset$ and $f(\langle 0,i\rangle) = i$,
is a surjective pp-morphism.
(Note that $f$ is not a p-morphism of dual spaces of Heyting
algebras.) 
Hence, by duality, the algebra $\mathrm{Up}(P_G)$ embeds in
$\mathrm{Up}(U)$, that is $\mathrm{Up}(\mathbf{2}\times I)$, which again by
duality is isomorphic to 
$\bigl(\mathrm{Up}(\mathbf{2})\bigr)^I$. But $\mathrm{Up}(\mathbf{2})$
is isomorphic to $\Bnalg{1}$, so we obtain that
$\mathrm{Up}(P_G)$ embeds in $\Bnalg{1}^{I}$.
To recover $G$ from $\mathrm{Up}(P_G)$ it suffices to note that
(i) it has precisely one atom, (ii) the elements $\up{\{u\}}$ are join-irreducible
covers of the unique atom, so we take as vertices of $G$ precisely these
elements, (iii)   
$\{u,v\}\in E$ if and only if $\up{\{u\}}$ and $\up{\{v\}}$ have a unique
join-irreducible element above them.
Clearly, all these properties are first-order expressible in the language
of p-algebras.
\end{proof}

Our proof method is illustrated in Figure~\ref{fig:method}. The graph $G$ and
the poset $P_G$ are on the left, the algebra $\mathrm{Up}(P_G)$ is
on the right. We write $u$ instead of $\{u\}$ and
$\overline{uv}$ instead of $\{u,v\}$ for conciseness.
\begin{figure}
\begin{center}
\begin{tikzpicture}
\node[dot, label=below:$u$] (u) at (0,0) {};
\node[dot, label=below:$v$] (v) at (0.7,0) {};
\node[dot, label=below:$w$] (w) at (1.4,0) {};
\path[draw,thick] (u)--(v);
\node (l) at (0.7, -1.7) {$G$};
\end{tikzpicture}
\qquad
\begin{tikzpicture}
\node[dot, label=above:$\emptyset$] (t) at (0.7,1.4) {};  
\node[dot, label=left:$u$] (u) at (0,0.7) {};
\node[dot, label=left:$v$] (v) at (0.7,0.7) {};
\node[dot, label=right:$w$] (w) at (1.4,0.7) {};
\node[dot, label=below:$\overline{uv}$] (uv) at (0.35, 0) {};
\path[draw,thick] (u)--(t) (v)--(t) (w)--(t) (u)--(uv) (v)--(uv);
\node (l) at (0.7, -1) {$P_G$};
\end{tikzpicture}
\qquad\qquad
\begin{tikzpicture}
\node[dot, label=below:$\emptyset$] (bot) at (0,0) {};
\node[bdot, label=right:$\up{\emptyset}$] (t) at (0,0.7) {};
\node[bdot, label=left:$\up{u}$] (u) at (-0.7,1.4) {};
\node[bdot, label=left:$\up{v}$] (v) at (0,1.4) {};
\node[bdot, label=right:$\up{w}$] (w) at (0.7,1.4) {};
\node[dot] (uv) at (-0.7,2.1) {};
\node[dot] (uw) at (0,2.1) {};
\node[dot] (vw) at (0.7,2.1) {};
\node[dot] (uvw) at (0,2.8) {};
\node[bdot, label=left:$\up{\overline{uv}}$] (uvi) at (-1.4,2.8) {};
\node[dot, label=above:$P_G$] (top) at (-0.7,3.5) {};
\path[draw,thick] (bot)--(t)--(u)--(uv)--(uvi)--(top)--(uvw)--(vw)--(w)--(t);
\path[draw,thick] (t)--(v)--(uv)--(uvw)--(uw)--(w) (u)--(uw) (v)--(vw);
\node (l) at (-2, -0.3) {$\mathrm{Up}(P_G)$};
\end{tikzpicture}
\end{center}\label{fig:method}\caption{Proof method illustrated.}
\end{figure}
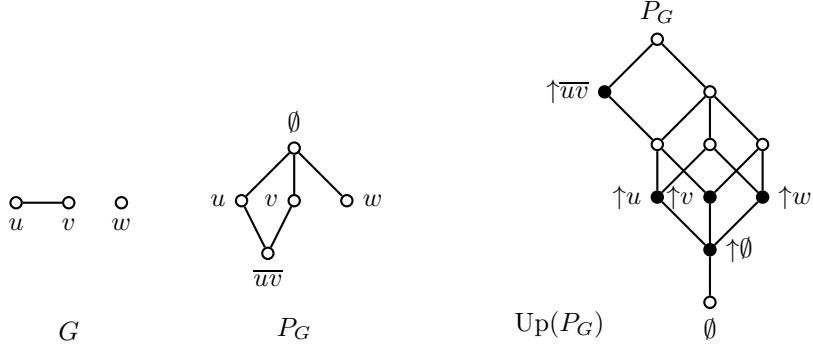
Note that $P_G$ above is a pp-morphic image of the disjoint union of four copies
of the two element chain. The black dots in the diagram of
$\mathrm{Up}(P_G)$ mark the join-irreducible elements: under the inverse of the
ordering inherited from the algebra they form a poset 
order-isomorphic to $P_G$. 

\begin{cor}\label{cor:fin-dec-qvar}
A quasivariety $\mathcal{Q}$ of p-algebras has decidable first-order theory
of finite algebras if and only if $\mathcal{Q}\subseteq\var{BA}$.
\end{cor}

Our proof did not use any properties of p-algebras except duality at the
finite level. This duality includes the usual Priestley duality, so the
technique 
applies without changes to any class of lattices closed under sublattices and
finite products, and containing a nontrivial member. Since any nontrivial
quasivariety of lattices is such a class, our method gives a new direct proof of
a well known result that all nontrivial varieties of lattices are finitely
undecidable. This result follows from~\cite{Idz89b}, and was also proved directly
by McKenzie, but not published (see~\cite{Kus02}).

\section{Final remarks and questions}

Let us first put Corollaries~\ref{cor:dec-qvar} and~\ref{cor:fin-dec-qvar} together,
to state the promised characterisation of decidable quasivarieties of
p-algebras.

\begin{theorem}
Let $\mathcal{Q}$ be a quasivariety of p-algebras.
The following are equivalent.
\begin{enumerate}
\item $\mathrm{Th}(\mathcal{Q})$ is decidable.
\item $\mathrm{Th}(\mathcal{Q}_{\mathrm{fin}})$ is decidable.
\item $\mathcal{Q}\subseteq \var{BA}$.
\item $\mathcal{Q}$ is either trivial or equal to $\var{BA}$.
\end{enumerate}  
\end{theorem}  

Since there are uncountably many
quasivarieties of p-algebras (cf. Wroński~\cite{Wro76} and
Gr\"atzer~\emph{et al.}~\cite{GLQ80}), uncountably many of their
quasi\-equational theories must be undecidable. It is known that
$Q(\Bnalg{i})$, for $i= 0,1,2$, are varieties without proper subsquasivarieties,
indeed, $Q(\Bnalg{i}) = \var{Pa}_i$. Hence, their quasiequational theories
coincide with equational theories, and these are decidable. As far as we know,
beyond that  there is no characterisation of quasivarieties with decidable
quasiequational theories, and no concrete and natural examples of
quasivarieties---necessarily containing $Q(\Bnalg{2})$---with undecidable
quasiequational theories.  

\begin{problem}
Construct a natural example of a quasivariety $Q$ of p-algebras such that
$Q(\Bnalg{2})\subseteq Q$ and $Q$ has undecidable
quasiequational theory. 
\end{problem}

\begin{problem}
Characterise quasivarieties of p-algebras with decidable
quasiequational theories. 
\end{problem}

It is known that all varieties of p-algebras have decidable equational theories.
Statman~\cite{Sta79} proved that the problem of determining whether an equation
holds in all Heyting algebras is PSPACE-complete, so the analogous problem
for p-algebras is certainly in PSPACE. On the other hand, the same
problem for Boolean algebras is coNP-complete.

\begin{problem}
What is the complexity of determining whether an equation holds in
all p-algebras? 
\end{problem}

\bibliographystyle{plain}
\bibliography{decidability}

\end{document}